%% file: alpha_p-4.tex
\documentclass[oneside,english]{amsart}
\usepackage[T1]{fontenc}
\usepackage{verbatim}
\usepackage{amstext}
\usepackage{amsthm}
\usepackage{amssymb}

\makeatletter
\numberwithin{equation}{section}
\numberwithin{figure}{section}
\theoremstyle{plain}
\newtheorem{thm}{\protect\theoremname}[section]
  \theoremstyle{remark}
  \newtheorem{rem}[thm]{\protect\remarkname}
  \theoremstyle{plain}
  \newtheorem{lem}[thm]{\protect\lemmaname}
  \theoremstyle{plain}
  \newtheorem{conjecture}[thm]{\protect\conjecturename}
  \theoremstyle{plain}
  \newtheorem{prop}[thm]{\protect\propositionname}
  \theoremstyle{definition}
  \newtheorem{example}[thm]{\protect\examplename}

\usepackage{pdfsync}

\input{packages_and_functions.for_preamble}

\makeatother

\usepackage{babel}
  \providecommand{\conjecturename}{Conjecture}
  \providecommand{\examplename}{Example}
  \providecommand{\lemmaname}{Lemma}
  \providecommand{\propositionname}{Proposition}
  \providecommand{\remarkname}{Remark}
\providecommand{\theoremname}{Theorem}

\begin{document}

\title{Notes on the motivic McKay correspondence for the group scheme $\alpha_{p}$}

\author{Fabio Tonini and Takehiko Yasuda}

\address{Freie Universit\"{a}t Berlin, FB Mathematik und Informatik, Arnimallee
3, Zimmer 112A Berlin, 14195 Germany}

\email{tonini@zedat.fu-berlin.de}

\address{Department of Mathematics, Graduate School of Science, Osaka University,
Toyonaka, Osaka 560-0043, Japan}

\email{takehikoyasuda@math.sci.osaka-u.ac.jp}
\begin{abstract}
We formulate a conjecture on the motivic McKay correspondence for
the group scheme $\alpha_{p}$ in characteristic $p>0$ and give a
few evidences. The conjecture especially claims that there would be
a close relation between quotient varieties by $\alpha_{p}$ and ones
by the cyclic group of order $p$. 
\end{abstract}

\maketitle
\global\long\def\AA{\mathbb{A}}
\global\long\def\PP{\mathbb{P}}
\global\long\def\NN{\mathbb{N}}
\global\long\def\GG{\mathbb{G}}
\global\long\def\ZZ{\mathbb{Z}}
\global\long\def\QQ{\mathbb{Q}}
\global\long\def\CC{\mathbb{C}}
\global\long\def\FF{\mathbb{F}}
\global\long\def\LL{\mathbb{L}}
\global\long\def\RR{\mathbb{R}}
\global\long\def\MM{\mathbb{M}}
\global\long\def\SS{\mathbb{S}}

\global\long\def\bx{\boldsymbol{x}}
\global\long\def\by{\boldsymbol{y}}
\global\long\def\bf{\mathbf{f}}
\global\long\def\ba{\mathbf{a}}
\global\long\def\bs{\mathbf{s}}
\global\long\def\bt{\mathbf{t}}
\global\long\def\bw{\mathbf{w}}
\global\long\def\bb{\mathbf{b}}
\global\long\def\bv{\mathbf{v}}
\global\long\def\bp{\mathbf{p}}
\global\long\def\bq{\mathbf{q}}
\global\long\def\bm{\mathbf{m}}
\global\long\def\bj{\mathbf{j}}
\global\long\def\bM{\mathbf{M}}
\global\long\def\bd{\mathbf{d}}
\global\long\def\bA{\mathbf{A}}
\global\long\def\bB{\mathbf{B}}
\global\long\def\bC{\mathbf{C}}
\global\long\def\bP{\mathbf{P}}
\global\long\def\bX{\mathbf{X}}
\global\long\def\bY{\mathbf{Y}}
\global\long\def\bZ{\mathbf{Z}}
\global\long\def\bW{\mathbf{W}}
\global\long\def\bV{\mathbf{V}}
\global\long\def\bU{\mathbf{U}}
\global\long\def\bN{\mathbf{N}}

\global\long\def\cN{\mathcal{N}}
\global\long\def\cW{\mathcal{W}}
\global\long\def\cY{\mathcal{Y}}
\global\long\def\cM{\mathcal{M}}
\global\long\def\cF{\mathcal{F}}
\global\long\def\cX{\mathcal{X}}
\global\long\def\cE{\mathcal{E}}
\global\long\def\cJ{\mathcal{J}}
\global\long\def\cO{\mathcal{O}}
\global\long\def\cD{\mathcal{D}}
\global\long\def\cZ{\mathcal{Z}}
\global\long\def\cR{\mathcal{R}}
\global\long\def\cC{\mathcal{C}}
\global\long\def\cL{\mathcal{L}}
\global\long\def\cV{\mathcal{V}}
\global\long\def\cU{\mathcal{U}}
\global\long\def\cS{\mathcal{S}}
\global\long\def\cT{\mathcal{T}}
\global\long\def\cA{\mathcal{A}}
\global\long\def\cB{\mathcal{B}}
\global\long\def\cG{\mathcal{G}}
\global\long\def\cP{\mathcal{P}}

\global\long\def\fs{\mathfrak{s}}
\global\long\def\fp{\mathfrak{p}}
\global\long\def\fm{\mathfrak{m}}
\global\long\def\fX{\mathfrak{X}}
\global\long\def\fV{\mathfrak{V}}
\global\long\def\fx{\mathfrak{x}}
\global\long\def\fv{\mathfrak{v}}
\global\long\def\fY{\mathfrak{Y}}
\global\long\def\fa{\mathfrak{a}}
\global\long\def\fb{\mathfrak{b}}
\global\long\def\fc{\mathfrak{c}}
\global\long\def\fO{\mathfrak{O}}
\global\long\def\fd{\mathfrak{d}}
\global\long\def\fP{\mathfrak{P}}

\global\long\def\rv{\mathbf{\mathrm{v}}}
\global\long\def\rx{\mathrm{x}}
\global\long\def\rw{\mathrm{w}}
\global\long\def\ry{\mathrm{y}}
\global\long\def\rz{\mathrm{z}}
\global\long\def\bv{\mathbf{v}}
\global\long\def\bw{\mathbf{w}}
\global\long\def\sv{\mathsf{v}}
\global\long\def\sx{\mathsf{x}}
\global\long\def\sw{\mathsf{w}}

\global\long\def\Spec{\mathrm{Spec}\,}
\global\long\def\Hom{\mathrm{Hom}}
\global\long\def\sht{\mathrm{sht}}

\global\long\def\Var{\mathbf{Var}}
\global\long\def\Gal{\mathrm{Gal}}
\global\long\def\Jac{\mathrm{Jac}}
\global\long\def\Ker{\mathrm{Ker}}
\global\long\def\Image{\mathrm{Im}}
\global\long\def\Aut{\mathrm{Aut}}
\global\long\def\diag{\mathrm{diag}}
\global\long\def\characteristic{\mathrm{char}}
\global\long\def\tors{\mathrm{tors}}
\global\long\def\sing{\mathrm{sing}}
\global\long\def\red{\mathrm{red}}
\global\long\def\ord{\mathrm{ord}}
\global\long\def\pt{\mathrm{pt}}
\global\long\def\op{\mathrm{op}}
\global\long\def\Val{\mathrm{Val}}
\global\long\def\Res{\mathrm{Res}}
\global\long\def\Pic{\mathrm{Pic}}
\global\long\def\disc{\mathrm{disc}}
\global\long\def\Coker{\mathrm{Coker}}
 \global\long\def\length{\mathrm{length}}
\global\long\def\sm{\mathrm{sm}}
\global\long\def\rank{\mathrm{rank}}
\global\long\def\age{\mathrm{age}}
\global\long\def\et{\mathrm{et}}
\global\long\def\hom{\mathrm{hom}}
\global\long\def\tor{\mathrm{tor}}
\global\long\def\reg{\mathrm{reg}}
\global\long\def\cont{\mathrm{cont}}
\global\long\def\crep{\mathrm{crep}}
\global\long\def\Stab{\mathrm{Stab}}
\global\long\def\discrep{\mathrm{discrep}}
\global\long\def\mld{\mathrm{mld}}
\global\long\def\GCov{G\textrm{-}\mathrm{Cov}}
\global\long\def\P{\mathrm{P}}

\global\long\def\GL{\mathrm{GL}}
\global\long\def\codim{\mathrm{codim}}
\global\long\def\prim{\mathrm{prim}}
\global\long\def\cHom{\mathcal{H}om}
\global\long\def\cSpec{\mathcal{S}pec}
\global\long\def\Proj{\mathrm{Proj}\,}
\global\long\def\modified{\mathrm{mod}}
\global\long\def\ind{\mathrm{ind}}
\global\long\def\rad{\mathrm{rad}}
\global\long\def\Conj{\mathrm{Conj}}
\global\long\def\fie{\textrm{-}\mathrm{fie}}
\global\long\def\NS{\mathrm{NS}}
\global\long\def\Disc{\mathrm{Disc}}
\global\long\def\hotimes{\hat{\otimes}}
\global\long\def\Fil{\mathrm{Fil}}
\global\long\def\Inn{\mathrm{Inn}}
\global\long\def\rfil{\mathrm{rfil}}
\global\long\def\per{\mathrm{per}}
\global\long\def\id{\mathrm{id}}
\global\long\def\AffVar{\mathbf{AffVar}}
\global\long\def\uni{\mathrm{uni}}
\global\long\def\Alg{\mathbf{Alg}}
\global\long\def\PSch{\textit{P-}\mathbf{Sch}}
\global\long\def\PQVar{\textit{P-}\mathbf{QVar}}
\global\long\def\QVar{\mathbf{QVar}}

\input{packages_and_functions.tex}

\makeatletter 
\providecommand\@dotsep{5} 
\makeatother 

\global\long\def\st{\mathrm{st}}

\section{Introduction}

The motivic McKay correspondence was established by Batyrev \cite{MR1677693}
and Denef\textendash Loeser \cite{MR1905024} in characteristic zero.
A version of this theory says that given a linear action of a finite
group $G$ on an affine space $\AA_{k}^{d}$ without pseudo-reflection,
we can express the motivic stringy invariant $M_{\st}(\AA_{k}^{d}/G)$
of the quotient variety $\AA_{k}^{d}/G$ as a sum of the form $\sum_{g\in\Conj(G)}\LL^{a(g)}$,
where $\Conj(G)$ is the set of conjugacy classes of $G$ and $a$
is a function on $\Conj(G)$ with values in $\frac{1}{\sharp G}\ZZ$.
This can be generalized to the tame case in characteristic $p>0$
(the case $p\nmid\sharp G$) without essential change (see \cite{MR2271984}).
After studying the case of the cyclic group of order $p$, the second
author formulated a conjectural generalization to the wild case (the
case $p\mid\sharp G$) (see \cite{MR3230848,wild-p-adic}). In this
conjecture, the sum $\sum_{g\in\Conj(G)}\LL^{a(g)}$ is replaced with
an integral of the form $\int_{\Delta_{G}}\LL^{a(g)}$, where $\Delta_{G}$
is the moduli space of $G$-torsors over the punctured formal disk
$\Spec k((t))$ and $a$ is a $\frac{1}{\sharp G}\ZZ$-valued function
on it. 

The aim of this paper is to discuss the case where $G$ is the group
scheme $\alpha_{p}$ rather than a genuine finite group, as a first
step towards further generalization to arbitrary finite group schemes.
We will formulate a conjectural expression (Conjecture \ref{conj:main})
for $M_{\st}(\AA_{k}^{d}/G)$ again of the form $\int_{\Delta_{G}}\LL^{a(g)}$
under the condition $D_{\bd}\ge2$ (for the definition of $D_{\bd}$,
see Section \ref{sec:conjecture}). But here we have a new phenomenon:
the moduli space $\Delta_{G}$ in this case is an ind-pro-limit of
finite dimensional spaces rather than an ind-limit as in the case
of genuine finite groups. Therefore we need to define a motivic measure
on $\Delta_{G}$ in terms of truncation maps as we do for the arc
space. Our conjecture also indicates a close relation between the
case of $G=\alpha_{p}$ and the case of the cyclic group $H:=\ZZ/p\ZZ$
of order $p$. There exists a one-to-one correspondence between $G$-representations
and $H$-representations. The conjecture says that if $\AA_{k}^{d}/G$
and $\AA_{k}^{d}/H$ are quotient varieties associated to $G$ and
$H$-representations corresponding to each other, then they have equal
motivic stringy invariant. We will give a few evidences for this conjecture.
Note that Hiroyuki Ito \cite{Ito-talk} earlier observed a similarity
between surface quotient singularities by (non-linear) $G$-actions
and $H$-actions. This was an inspiration for our conjecture. 

The paper is organized as follows. In Section \ref{sec:Moduli}, we
describe the moduli spaces $\Delta_{G}$ and $\Delta_{H}$. In Section
\ref{sec:Motivic-measures}, we define motivic measures on these spaces.
In Section \ref{sec:conjecture}, we formulate our main conjecture.
In Section \ref{sec:Two-Examples}, we give two examples supporting
the equality of stringy invariants of $\AA_{k}^{d}/G$ and $\AA_{k}^{d}/H$.
In Section \ref{sec:two-dim}, we discuss the case where $G$ acts
on $\AA_{k}^{2}$, that is, the case $D_{\bd}=1$, as a toy model
and show the change of variables formula for the quotient map $\AA_{k}^{2}\to\AA_{k}^{2}/G$.
This would be viewed as a supporting evidence for our conjecture that
$M_{\st}(\AA_{k}^{d}/G)$ is expressed as $\int_{\Delta_{G}}\LL^{a(g)}$
when $D_{\bd}\ge2$. In Appendix, we briefly recall the representation
theory of $\alpha_{p}$ in terms of relations to nilpotent endomorphisms
and derivations.

In what follows, we work over an algebraically closed field $k$ of
characteristic $p>0$. We always denote by $G$ the group scheme $\alpha_{p}$
and by $H$ the cyclic group of order $p$. We write the coordinate
ring of $\alpha_{p}$ as $k[\epsilon]=k[x]/(x^{p})$ with $\epsilon$
the image of $x$ in this quotient ring. 

The authors would like to thank Hiroyuki Ito for helpful discussion.
The second author was supported by JSPS KAKENHI Grant Numbers JP15K17510
and JP16H06337. 

\section{\label{sec:Moduli}Moduli of $G$ and $H$-torsors over the punctured
formal disk.}

The group scheme $G$ fits into the exact sequence
\[
0\to G\to\GG_{a}\xrightarrow{F}\GG_{a}\to0,
\]
where $F$ is the Frobenius map. Therefore the $G$-torsors over $\Spec k((t))$
are parameterized by $k((t))/F(k((t)))$. Let 
\[
\Delta_{G}:=\left\{ \sum_{i\in\ZZ;\,p\nmid i}c_{i}t^{i}\mid c_{i}\in k\right\} \subset k((t)),
\]
the set of Laurent power series having only terms of degree prime
to $p$. The inclusion map $\Delta_{G}\to k((t))$ induces a bijection
$\Delta_{G}\to k((t))/F(k((t)))$. Thus we regard $\Delta_{G}$ as
the ``moduli space'' of $G$-torsors over $\Spec k((t))$. The $G$-torsor
corresponding to a Laurent power series $f\in\Delta_{G}$ is $\Spec k((t))[z]/(z^{p}-f)$
and the action of $\alpha_{p}$ is defined so that the associated
coaction is the $k((t))$-algebra homomorphism
\[
\frac{k((t))[z]}{(z^{p}-f)}\to\frac{k((t))[z]}{(z^{p}-f)}[\epsilon],\,z\mapsto z+\epsilon.
\]

We can make a similar construction for $H$-torsors. Let $\wp\colon k((t))\to k((t))$
be the Artin-Schreier map given by $\wp(f)=f^{p}-f$ and let 
\[
\Delta_{H}:=\left\{ \sum_{i\in\ZZ;\,i<0,\,p\nmid i}c_{i}t^{i}\mid c_{i}\in k\right\} \subset k((t)),
\]
the set of Laurent polynomials having only terms of \emph{negative
}degree prime to $p$. Then the $H$-torsors are parameterized by
$k((t))/\wp(k((t)))$ and the composite map $\Delta_{H}\hookrightarrow k((t))\twoheadrightarrow k((t))/\wp(k((t)))$
is bijective. Thus we regard $\Delta_{H}$ as the ``moduli space''
of $H$-torsors over $\Spec k((t))$. The $H$-torsor corresponding
to $f\in\Delta_{H}$ is $\Spec k((t))[z]/(z^{p}-z-f)$ where a generator
of $H$ acts by $z\mapsto z+1$. 
\begin{rem}
Constructing the true moduli space (stack) which represents a relevant
moduli functor is a more difficult problem. However, the above ad
hoc version would be sufficient for our application to motivic integration,
since we work over an algebraically closed field. When $k$ is algebraically
closed, the coarse moduli space for $\Delta_{H}$ was constructed
by Harbater \cite{MR579791}. The fine moduli stack for $\Delta_{H}$
over an arbitrary field of characteristic $p>0$ was constructed by
the authors \cite{Tonini:2017qr}. The moduli space $\Delta_{G}$
has not been seriously studied yet, as far as the authors know.
\end{rem}

\section{Motivic measures on $\Delta_{G}$ and $\Delta_{H}$\label{sec:Motivic-measures}}

For a positive integer $j$ prime to $p$, the set $\Delta_{H}^{\ge-j}:=\{f\in\Delta_{H}\mid\ord(f)\ge-j\}$
is the affine space $\A_{k}^{j-\lfloor j/p\rfloor}$. Thus $\Delta_{H}$
is the union of affine spaces $\Delta_{H}^{\ge-j}$. We say that a
subset $C$ of $\Delta_{H}$ is \emph{constructible }if $C$ is a
constructible subset of some $\Delta_{H}^{\ge-j}$. We define the
\emph{motivic measure} $\mu_{H}$ on $\Delta_{H}$ by $\mu_{H}(C):=[C]$
say in $\hat{\cM}'$, a version of the complete Grothendieck ring
of varieties used in \cite{MR3230848}. (In this note, we do not discuss
what additional relation would be really necessary to put on the complete
Grothendieck ring for the McKay correspondence in the case of the
group scheme $G$. This should be clarified in a future study.)

For $n\in\ZZ$, let 
\[
\tau_{n}\colon\Delta_{G}\to\Delta_{G,n}:=\frac{\Delta_{G}}{\Delta_{G}\cap t^{np}k[[t]]}
\]
be the quotient map, which truncates the terms of degrees $\ge np$.
We often identify $\Delta_{G,0}$ with $\Delta_{H}$ through the natural
bijection $\Delta_{H}\hookrightarrow\Delta_{G}\to\Delta_{G,0}$ and
$\tau_{0}$ with a map $\Delta_{G}\to\Delta_{H}$. When $n\ge0$,
$\Delta_{G,n}$ is the product of $\Delta_{H}$ and an affine space
$\AA_{k}^{n(p-1)}$. We say that a subset $C\subset\Delta_{G}$ is
a \emph{cylinder of level $n$ }if $\tau_{n}(C)$ is a constructible
subset of $\Delta_{G,n}$ and $C=\tau_{n}^{-1}(\tau_{n}(C))$. For
a cylinder $C\subset\Delta_{G}$ of level $n$, we define its measure
as 
\[
\mu_{G}(C):=[\tau_{n}(C)]\LL^{-n(p-1)}\in\hat{\cM}'.
\]
Since the natural map $\Delta_{G,n+1}\to\Delta_{G,n}$ is the trivial
$\A_{k}^{p-1}$-fibration, the element $[\tau_{n}(C)]\LL^{-n(p-1)}$
is independent of the choice of a sufficiently large $n$. For instance,
$\Delta_{G}^{\ge0}:=\{f\in\Delta_{G}\mid\ord(f)\}$ is a cylinder
of level zero such that $\tau_{0}(\Delta_{G}^{\ge0})$ is a singleton.
Note that this set contains $0\in k((t))$, according to the convention
$\ord(0)=+\infty$. The measure of $\Delta_{G}^{\ge0}$ is
\[
\mu_{G}(\Delta_{G}^{\ge0})=[1\,\pt]=1.
\]

\section{A conjecture on the McKay correspondence for linear $G$-actions\label{sec:conjecture}}

To a sequence of integers, $\bd=(d_{1},d_{2},\dots,d_{l})$ such that
$1\le d_{\lambda}\le p$ and $d_{\lambda}\ge d_{\lambda+1}$, we associate
the linear representation $W$ of $H$ over $k$ that have $d_{\lambda}$-dimensional
indecomposables as direct summands. Namely a generator of $H$ acts
on the vector space $W=k^{d}$, $d:=|\bd|=\sum_{\lambda=1}^{l}d_{\lambda}$
by a matrix whose Jordan normal form has Jordan blocks of sizes $d_{\lambda}$
with diagonal entries 1. The map $\bd\mapsto W$ gives a one-to-one
correspondence between sequences $\bd$ of integers as above and isomorphism
classes of $H$-representions over $k$. 

Similarly, to a sequence $\bd$ as above, we associate also the linear
representation $V$ of $G$ over $k$ that have $d_{\lambda}$-dimensional
indecomposables as direct summands. If $\rho$ denotes the nilpotent
linear endomorphism of $k[x_{1},\dots,x_{d}]$ defined by Jordan blocks
of sizes $d_{\lambda}$ with diagonal entries 0, then the map 
\[
k[x_{1},\dots,x_{d}]\to k[x_{1},\dots,x_{d}][\epsilon],\,f\mapsto\sum_{i=0}^{p-1}\frac{\rho^{i}(f)}{i!}\epsilon^{i}
\]
defines the linear $G$-action on $\AA_{k}^{d}$. For details, see
Appendix. In particular, through sequences $\bd$, we get a one-to-one
correspondence between $G$-representations and $H$-representations.
We expect that the quotient varieties $V/G$ and $W/H$ for the same
$\bd$ would be very similar in the sense we will make more precise. 

Let us fix a sequence $\bd$ as above. Following \cite{MR3230848},
for a positive integer $j$ prime to $p$, we define 
\[
\sht(j):=\sum_{\lambda=1}^{l}\sum_{i=1}^{d_{\lambda}-1}\left\lfloor \frac{ij}{p}\right\rfloor .
\]
We define a function $\sht\colon\Delta_{H}\to\ZZ$ by 
\[
\sht(f):=\begin{cases}
\sht(-\ord(f)) & (f\ne0)\\
0 & (f=0).
\end{cases}
\]
We define another function $\sht'$ on $\Delta_{H}$ by 
\[
\sht'(f):=\begin{cases}
\sht(f)-l & (f\ne0)\\
-d & (f=0).
\end{cases}
\]
Fibers of these functions are constructible subsets. In general, for
a function $u\colon\Delta_{H}\to\ZZ$ with constructible fibers, we
define the motivic integral
\[
\int_{\Delta_{H}}\LL^{u}\,d\mu_{H}:=\sum_{i\in\ZZ}[u^{-1}(i)]\LL^{i}\in\hat{\cM}',
\]
provided that the last infinite sum converges in $\hat{\cM}'$. When
it diverges, we formally put $\int_{\Delta_{H}}\LL^{f}\,d\mu_{H}:=\infty$. 

We define a numerical invariant $D_{\bd}:=\sum_{\lambda=1}^{l}(d_{\lambda}-1)d_{\lambda}/2$,
which we also think of as invariants of representations $V_{\bd}$
and $W_{\bd}$. Integrals 
\[
\int_{\Delta_{H}}\LL^{-\sht}\,d\mu_{H},\quad\int_{\Delta_{H}}\LL^{-\sht'}\,d\mu_{H}
\]
converge exactly when $D_{\bd}\ge p$ (for details of computation,
see \cite{MR3230848}). If $D_{\bd}=0$, then the corresponding $H$-action
is trivial. If $D_{\bd}=1$, then $\bd=(2,1,\dots,1)$ , $W/H\cong\A_{k}^{d}$
and the quotient map $W\to W/H$ has ramification locus of codimension
one. Therefore the case $D_{\bd}\ge2$ is of our main interest, although
the case $D_{\bd}=1$ will be discussed in Section \ref{sec:two-dim}
as a toy case. 

The quotient variety $W/H$ is factorial \cite[Th. 3.8.1]{MR2759466},
in particular, has the invertible canonical sheaf $\omega_{W/H}$.
The \emph{$\omega$-Jacobian ideal sheaf} $\cJ\subset\cO_{W/H}$ is
defined by $\cJ\omega_{W/H}:=\Image(\bigwedge^{|\bd|}\Omega_{W/H}\to\omega_{W/H})$.
Let $o\in W/H$ be the image of the origin of $W$, $J_{\infty}(W/H)$
the arc space of $W/H$ and $J_{\infty}(W/H)_{o}$ the preimage of
$o$ by the natural map $J_{\infty}(W/H)\to W/H$. The\emph{ motivic
stringy invariant} (resp. the \emph{motivic stringy invariant at $o$})
of $W/H$ is defined to be 
\begin{gather*}
M_{\st}(W/H):=\int_{J_{\infty}(W/H)}\LL^{\ord\cJ}\,d\mu_{W/H}\\
\left(\text{resp. }M_{\st}(W/H)_{o}:=\int_{J_{\infty}(W/H)_{o}}\LL^{\ord\cJ}\,d\mu_{W/H}\right).
\end{gather*}
If there exists a crepant resolution $\phi\colon Y\to W/H$, then
we have $M_{\st}(W/H)=[Y]$ and $M_{\st}(W/H)_{o}=[\phi^{-1}(o)]$. 

The following theorem proved in \cite{MR3230848} can be considered
as the motivic McKay correspondence for linear $H$-actions: 
\begin{thm}
\label{thm:McKay-H}If $D_{\bd}\ge2$, we have the following equalities
in $\hat{\cM}'\cup\{\infty\}$, 
\[
M_{\st}(W/H)_{o}=\int_{\Delta_{H}}\LL^{-\sht}\,d\mu_{H},\quad M_{\st}(W/H)=\int_{\Delta_{H}}\LL^{-\sht'}\,d\mu_{H}.
\]
\end{thm}

Since the convergence of $M_{\st}(W/H)$ (or $M_{\st}(M/H)_{o}$)
is equivalent to that $W/H$ has only canonical singularities, this
theorem in particular implies that $W/H$ has canonical singularities
if and only if $D_{\bd}\ge p$ (see \cite{Yasuda:2017gd}). 

In the rest of this section, we will formulate a conjecture for $V/G$
similar to this theorem. For this purpose, we extend functions $\sht,\sht'\colon\Delta_{H}\to\ZZ$
to functions on $\Delta_{G}$ by
\begin{gather*}
\sht(f)=\sht(\tau_{0}(f))=\begin{cases}
\sht(-\ord(f)) & (\ord(f)<0)\\
\sht(0) & (\ord(f)\ge0)
\end{cases},\\
\sht'(f):=\sht'(\tau_{0}(f))=\begin{cases}
\sht'(-\ord(f)) & (\ord(f)<0)\\
\sht'(0) & (\ord(f)\ge0)
\end{cases}.
\end{gather*}
Each fiber of the extended functions $\sht$ or $\sht'$ is a cylinder
of level zero and 
\begin{equation}
\tau_{0}(\sht^{-1}(i))=\left(\sht|_{\Delta_{H}}\right)^{-1}(i),\quad\tau_{0}((\sht')^{-1}(i))=\left(\sht^{'}|_{\Delta_{H}}\right)^{-1}(i).\label{eq:tau-sht}
\end{equation}
For a function $u\colon\Delta_{G}\to\ZZ$ whose fibers are cylinders,
we define 
\[
\int_{\Delta_{G}}\LL^{u}\,d\mu_{G}:=\sum_{i\in\ZZ}\mu_{G}(u^{-1}(i))\LL^{i}.
\]
From (\ref{eq:tau-sht}), we have
\[
\int_{\Delta_{G}}\LL^{-\sht}\,d\mu_{G}=\int_{\Delta_{H}}\LL^{-\sht}\,d\mu_{H},\quad\int_{\Delta_{G}}\LL^{-\sht'}\,d\mu_{H}=\int_{\Delta_{H}}\LL^{-\sht'}\,d\mu_{H}.
\]
\begin{lem}
The scheme $V/G$ is of finite type over $k$ and factorial. 
\end{lem}

\begin{proof}
Let $k[V]$ and $k[V/G]$ be the coordinate rings of $V$ and $V/G$
respectively, so that $k[V/G]=k[V]^{D}$. In particular $k[V]^{p}\subseteq k[V/G]$,
which easily implies that $V\arr V/G$ is a homeomorphism. Moreover
$V/G\arr V$ is finite, which implies that $V/G$ is of finite type
over $k$ because $k$ is perfect.

Let $P$ be a prime of height one of $k[V/G]$. We must show that
$P$ is principal. We have that $P$ is the restriction of a height
one prime ideal of $k[V]$, which is therefore generated by an irreducible
polynomial $f$. If $f\mid D(f)$ then, since $\deg(D(f))\le\deg(f)$,
we have $D(f)=cf$ for some $c\in k$. In particular $0=D^{p}(f)=c^{p}f$
so that $c=0$ and $D(f)=0$. In this case it follows easily that
$P=k[V/G]\cap(fk[V])=fk[V/G]$. 

So assume that $f\nmid D(f)$. We claim that $P=f^{p}k[V/G]$. Let
$x\in P-\{0\}$, that is $x=hf^{l}$ with $l>0$, $h$ coprime with
$f$ and $D(x)=0$. We have
\[
0=D(x)=D(h)f^{l}+hlf^{l-1}D(f)\then p\mid l\text{ and }D(h)=0
\]
 So $x=h(f^{p})^{l/p}\in f^{p}k[V/G]$.
\end{proof}
Thanks to this lemma, we can define the $\omega$-Jacobian ideal on
$V/G$ similarly to the case of $W/H$. In turn, we can define $M_{\st}(V/G),M_{\st}(V/G)_{o}$.
The following is our main conjecture.
\begin{conjecture}
\label{conj:main}If $D_{\bd}\ge2$, we have the following equalities
in $\hat{\cM}'\cup\{\infty\}$, 
\begin{gather*}
M_{\st}(V/G)_{o}=\int_{\Delta_{G}}\LL^{-\sht_{\bd}}\,d\mu_{G}\left(=M_{\st}(W/H)_{o}\right),\\
M_{\st}(V/G)=\int_{\Delta_{G}}\LL^{-\sht_{\bd}'}\,d\mu_{G}\left(=M_{\st}(W/H)\right).
\end{gather*}
\end{conjecture}

\section{Two Examples\label{sec:Two-Examples}}

In this section, we see two examples supporting the conjecture that
$M_{\st}(V/G)=M_{\st}(W/H)$ and $M_{\st}(V/G)_{o}=M_{\st}(W/H)_{o}$. 

\subsection{}

We first consider the case $\bd=(3)$, supposing $p\ge3$. If $p=3$,
then $W/H$ has a crepant resolution $\phi\colon U\to W/H$ such that
\begin{gather*}
M_{\st}(W/H)=[U]=\LL^{3}+2\LL^{2},\\
M_{\st}(W/H)_{o}=[\phi^{-1}(o)]=2\LL+1.
\end{gather*}
If $p>3$, then $W/H$ is not log canonical, in particular, $M_{\st}(W/H)=M_{\st}(W/H)_{o}=\infty$.
See \cite[Example 6.23]{MR3230848}. 

As for the $G$-action on $V$, the corresponding derivation $D$
on the coordinate ring $k[x,y,z]$ is given by $D(x)=0$, $D(y)=x$,
$D(z)=y$. We can compute 
\[
k[V/G]=k[x,y^{p},z^{p},y^{2}-2xz]\cong k[X,Y,Z,W]/(Y^{2}-W^{p}-X^{p}Z),
\]
where $X,Y,Z,W$ correspond to $2x,y^{p},z^{p},y^{2}-2xz$ respectively.
By simple computation of blowups, we can easily see that if $p=3$,
then $V/G$ has a crepant resolution. Using this resolution, we see
that $M_{\st}(V/G)=\LL^{3}+2\LL^{2}$ and $M_{\st}(V/G)_{o}=2\LL+1$.
We also see that if $p>3$, then $V/G$ is not canonical and $M_{\st}(V/G)=\infty$.
More details of computation are as follows.

Let us compute a (partial) resolution of $U_{0}:=V/G$. The singular
locus $U_{0,\sing}$ of $U_{0}$ is the affine line defined by $X=Y=W=0$.
Let $U_{1}\to U_{0}$ be the blowup of $U_{0}$ along $U_{0,\sing}$.
This is a crepant morphism and the exceptional locus is the trivial
$\PP^{1}$-bundle over $U_{0,\sing}$. 

When $p=3$, then $U_{1,\sing}$ is again an affine line and $U_{1,\sing}$
has an affine open neighborhood isomorphic to $\Spec k[X,Y,Z]/(Y^{2}+XZ)\times\AA_{k}^{1}$.
Therefore the blowup $U_{2}\to U_{1}$ along $U_{1,\sing}$ is a crepant
resolution and its exceptional locus is again the trivial $\PP^{1}$-bundle
over $U_{1,\sing}$. Let $\phi\colon U_{2}\to U_{0}$ be the natural
morphism, which is a crepant resolution. Then, the above computation
shows that 
\begin{gather*}
M_{\st}(V/G)=[U_{2}]=\LL^{3}+2\LL^{2},\\
M_{\st}(V/G)_{o}=[\phi^{-1}(o)]=2\LL+1.
\end{gather*}

When $p>3$, we claim that $U_{1}$ is not log canonical, and neither
is $U_{0}=V/G$, since $U_{1}\to U_{0}$ is crepant. By an explicit
computation, we see that $U_{1}$ has an affine open subset $U_{1}'$
which is isomorphic to a hypersurface defined by $Y^{2}-X^{p-2}W^{p}-X^{p-2}Z=0$.
Its singular locus $U_{1,\sing}'$ is defined by $X=Y=0$. Let $U_{2}\to U_{1}'$
be the blowup along $U_{1,\sing}'$ and $B\to\AA_{k}^{4}$ the blowup
of the ambient affine space along the same locus $U_{1,\sing}'$ so
that $U_{2}$ is a closed subset of $B$. Let $E\subset B$ be the
exceptional divisor of $B\to\AA_{k}^{4}$. We see
\[
K_{U_{2}/U_{1}'}=-E|_{U_{2}}.
\]
If $p=5$, then $U_{2}$ is normal and $-E|_{U_{2}}=-2E'$ for a prime
divisor $E'$. Thus $E'$ has discrepancy $-2$ over $U_{1}$, hence
$U_{1}$ is not log canonical. If $p>5$, then $U_{2}$ is not normal.
We similarly take the blowup $\psi\colon U_{3}\to U_{2}'$ of an affine
open subset $U_{2}'\subset U_{2}$ such that $K_{U_{3}/U_{2}'}=-F|_{U_{3}}$
and $\psi^{*}(E|_{U_{2}})=F|_{U_{3}}$, where $F$ is the exceptional
divisor of the blowup of the ambient affine space. We conclude that
$K_{U_{3}/U_{1}'}=-2F|_{U_{3}}$. This shows that $U_{1}$ is not
log canonical. Thus $V/G$ is not log canonical and we have $M_{\st}(V/G)=\infty$.
Let $T\to V/G$ be a log resolution on which the above exceptional
divisor with discrepancy $<-1$ appears. This exceptional divisor
on $T$ surjects onto the singular locus of $V/G$. This shows that
$M_{\st}(V/G)_{o}=\infty$. 
\begin{rem}
When $p=3$, the quotient singularities above by $H$ and $G$ are
the same as the two hypersurface singularities in characteristic three
in \cite[Th. 3]{MR2995023} up to suitable coordinate transforms.
\end{rem}

\subsection{}

Next we consider the case $p=2$ and $\bd=(2,2)$. Then $W/H$ is
the symmetric product of two copies of $\AA_{k}^{2}$. It has a crepant
resolution $\phi\colon U\to W/H$ constructed simply by blowing up
the singular locus once. This resolution coincides with the Hilbert
scheme of two points of $\AA_{k}^{2}$, usually denoted by $\mathrm{Hilb}^{2}(\AA_{k}^{2})$.
From an explicit description, we know that 
\begin{gather*}
M_{\st}(W/H)=[U]=\LL^{4}+\LL^{3},\\
M_{\st}(W/H)_{o}=[\phi^{-1}(o)]=\LL+1.
\end{gather*}

The corresponding derivation on the polynomial ring $k[x_{0},y_{0},x_{1},y_{1}]$
is given by $D(x_{i})=0$, $D(y_{i})=x_{i}$. The coordinate ring
of $V/G$ is 
\[
k[x_{0},y_{0}^{2},x_{1},y_{1}^{2},x_{0}y_{1}+x_{1}y_{0}]\cong k[V,W,X,Y,Z]/(Z^{2}+V^{2}Y+X^{2}W).
\]
The singular locus of $V/G$ is defined by $X=V=Z=0$ and isomorphic
to an affine plane. Let $\phi\colon U\to V/G$ be the blowup along
the singular locus. This is a crepant resolution such that the exceptional
locus is the trivial $\PP^{1}$-bundle over $(V/G)_{\sing}\cong\AA_{k}^{2}$.
Thus $[U]=\LL^{4}+\LL^{3}$ and $[\phi^{-1}(o)]=\LL+1$. 

\section{The change of variables in dimension two\label{sec:two-dim}}

In this section, we present some computation supporting the equalities
$M_{\st}(V/G)_{o}=\int_{\Delta_{G}}\LL^{-\sht_{\bd}}\,d\mu_{G}$ and
$M_{\st}(V/G)=\int_{\Delta_{G}}\LL^{-\sht_{\bd}'}\,d\mu_{G}$ in Conjecture
\ref{conj:main}. However, when $D_{\bd}\ge2$, $V/G$ and $W/H$
have singularities, which make analysis more difficult. Therefore
we consider the case $D_{\bd}=1$ as a toy model. We will use jet
schemes and the theory of integration above those spaces. For generalities
see \cite[Section 4]{MR3230848}. 

Then $\bd$ is of the form $(2,1,\dots,1)$, but it is enough to consider
the special case $\bd=(2)$, because there is no essential difference
in the general case. When $\bd=(2)$, if we write $W=\Spec k[x,y]$
and let a generator of $H$ act on it by $y\mapsto y$, $x\mapsto x+y$,
then $W/H=\Spec k[x^{p}-xy^{p-1},y]\cong\AA_{k}^{2}$. The quotient
map $W\to W/H$ is ramified along the divisor $y$, in particular,
the map is not crepant. Because of this, we do not have the equalities
in Conjecture \ref{conj:main} in this case. Instead we have
\begin{gather*}
\LL^{2}=M_{\st}(W/H)=\int_{J_{\infty}(W/H)}1\,d\mu_{W/G}=\int_{\cJ_{\infty}\cW}\LL^{-\ord(y^{p-1})-\fs}\,d\mu_{\cW},\\
1=M_{\st}(W/H)_{o}=\int_{(J_{\infty}(W/H))_{o}}1\,d\mu_{W/G}=\int_{(\cJ_{\infty}\cW)_{o}}\LL^{-\ord(y^{p-1})-\fs}\,d\mu_{\cW}
\end{gather*}
the last integral of which we will explain now. The domain of integral,
$\cJ_{\infty}\cW$, is the space of twisted arcs of the quotient stack
$\cW=[W/H]$. The use of the stack $\cW$ is only for this conventional
notation and not really necessary. We can describe this space as
\[
\cJ_{\infty}\cW:=\bigsqcup_{f\in\Delta_{H}}\Hom^{H}(\Spec\cO_{f},W)/H,
\]
where $\Spec\cO_{f}$ is the normalization of $\Spec k[[t]]$ in the
$H$-torsor over $\Spec k((t))$ corresponding to $f$, $\Hom^{H}(-,-)$
is the set of $H$-equivariant morphisms. The $(\cJ_{\infty}\cW)_{o}$
is the subset of $\cJ_{\infty}\cW$ consisting of $H$-orbits of $H$-equivariant
maps $\Spec\cO_{f}\to W$ sending the closed point(s) onto the origin
of $W$, and $(J_{\infty}(W/H))_{o}$ is the set of arcs $\Spec k[[t]]\to W/G$
sending the closed point to $o$. We can define a motivic measure
on $\cJ_{\infty}\cW$ and there exists a natural map $\cJ_{\infty}\cW\to J_{\infty}(W/H)$
which is almost bijective (bijective outside measure zero subsets)
and induces an almost bijection $(\cJ_{\infty}\cW)_{o}\to(J_{\infty}(W/H))_{o}$.
The $\ord(y^{p-1})$ is the function on $\cJ_{\infty}\cW$ assigning
orders of $y^{p-1}$ along twisted arcs. Note that the ideal $(y^{p-1})\subset k[x,y]$
is the Jacobian ideal of the map $W\to W/H$. Finally $\fs$ is the
composition of the natural map $\cJ_{\infty}\cW\to\Delta_{H}$ and
$\sht\colon\Delta_{H}\to\ZZ$. The equality $\int_{J_{\infty}(W/H)}1\,d\mu_{W/H}=\int_{\cJ_{\infty}\cW}\LL^{-\ord(y^{p-1})-\fs}\,d\mu_{\cW}$
can be thought of as the change of variables formula for the map $\cJ_{\infty}\cW\to J_{\infty}(W/H)$.
More generally, for a measurable function $F\colon C\to\ZZ$ on a
subset $C\subset J_{\infty}(W/G)$, if $\phi\colon\cJ_{\infty}\cW\to J_{\infty}(W/G)$
denotes the natural map, then
\begin{equation}
\int_{C}\LL^{F}\,d\mu_{W/H}=\int_{\phi^{-1}(C)}\LL^{F\circ\phi-\ord(y^{p-1})-\fs}\,d\mu_{\cW}.\label{eq:ch-vars-1}
\end{equation}
The term $-\ord(y^{p-1})$ corresponds to the ramification divisor
of $W\to W/H$, the divisor defined by the Jacobian ideal, or to the
relative canonical divisor of the proper birational map $\cW\to W/H$.
If $D_{\bd}\ge2$, then $W\to W/H$ is \'{e}tale in codimension one
and has no ramification divisor, but $W/H$ acquires singularities
as compensation. Therefore the corresponding formula in that case
is
\begin{equation}
\int_{C}\LL^{F+\cJ}\,d\mu_{W/H}=\int_{\phi^{-1}(C)}\LL^{F\circ\phi-\fs}\,d\mu_{\cW}.\label{eq:ch-vars-gen}
\end{equation}
When $C=J_{\infty}(W/H)$ and $F\equiv0$, then the right hand side
becomes
\[
\int_{\cJ_{\infty}\cW}\LL^{-\fs}\,d\mu_{\cW}=\int_{\Delta_{H}}\LL^{-\sht'}\,d\mu_{H}.
\]

We show a similar formula in the case of $G$. We first introduce
a counterpart of $\cJ_{\infty}\cW$. For $f\in\Delta_{G}$, let $\Spec K_{f}\to\Spec k((t))$
be the corresponding $G$-torsor. The underlying topological space
of $\Spec K_{f}$ is always a singleton.
\begin{lem}
If $f\ne0$, then $K_{f}$ is reduced, equivalently, $K_{f}$ is a
field. 
\end{lem}

\begin{proof}
The $K_{f}$ is a finite extension of $k((t))$ of degree $p$. If
$K_{f}$ is non-reduced, then the associated reduced ring $(K_{f})_{\red}$
is an extension of degree one, hence the natural map $k((t))\to(K_{f})_{\red}$
is an isomorphism. This means that the $G$-torsor $\Spec K_{f}\to\Spec k((t))$
admits a section, hence it is a trivial torsor and $f=0$. 
\end{proof}
For $f\ne0$, let $\Spec\cO_{f}$ be the normalization of $\Spec k[[t]]$
in $\Spec K_{f}$. For $f=0$, we define $\cO_{f}:=k[[t]][z]/(z^{p})$.
We say that a morphism $\Spec\cO_{f}\to V$ is $G$-equivariant if
the composition map $\Spec K_{f}\to\Spec\cO_{f}\to V$ is $G$-equivariant.
Note that the $G$-action on $\Spec K_{f}$ does not generally extend
to $\Spec\cO_{f}$ , which is the reason that we define $G$-equivariant
morphisms $\Spec\cO_{f}\to V$ in this way. Let $\Hom^{G}(\Spec K_{f},V)$
and $\Hom^{G}(\Spec\cO_{f},V)$ be the set of $G$-equivariant morphisms
$\Spec K_{f}\to V$ and $\Spec\cO_{f}\to V$ respectively. Regarding
$\Hom(\Spec\cO_{f},V)$ as a subset of $\Hom(\Spec K_{f},V)$, we
have 
\[
\Hom^{G}(\Spec\cO_{f},V)=\Hom(\Spec\cO_{f},V)\cap\Hom^{G}(\Spec K_{f},V).
\]
We then put 
\[
\cJ_{\infty}\cV:=\bigsqcup_{f\in\Delta_{G}}\Hom^{G}(\Spec\cO_{f},V).
\]
We now give an explicit description of this set. Let $D$ denote the
derivation on $k[x,y]$ given by $D(y)=x$, $D(x)=0$, which corresponds
to the $G$-action on $V=\Spec k[x,y]$. The $G$-action on $V$ is
given by the coaction:
\begin{align*}
\theta\colon k[x,y] & \to k[x,y][\epsilon]\\
x & \mapsto x\\
y & \mapsto y+x\epsilon
\end{align*}
The $G$-action on $\Spec K_{f}$ with $K_{f}=k((t))[z]/(z^{p}-f)$
is given by:
\begin{align*}
\psi_{f}\colon K_{f} & \to K_{f}[\epsilon]\\
z & \mapsto z+\epsilon
\end{align*}
An element of $K_{f}$ is uniquely written as $\sum_{i=0}^{p-1}a_{i}z^{i}$,
$a_{i}\in k((t))$ and a map $\gamma\colon\Spec K_{f}\to V$ is uniquely
determined by two elements $\gamma^{*}(x)=\sum_{i=0}^{p-1}a_{i}z^{i}$
and $\gamma^{*}(y)=\sum_{i=0}^{p-1}b_{i}z^{i}$ of $K_{f}$. The map
$\gamma$ is $G$-equivariant if and only if 
\[
\psi_{f}(\gamma^{*}(x))=(\gamma^{*}\otimes\id_{k[\epsilon]})(\theta(x)),\,\psi_{f}(\gamma^{*}(y))=(\gamma^{*}\otimes\id_{k[\epsilon]})(\theta(y)).
\]
The left equality is explicitly written as
\[
\sum_{i}a_{i}(z+\epsilon)^{i}=\sum_{i}a_{i}z^{i},
\]
which is equivalent to saying that $a_{i}=0$ for $i>0$. The right
equality then says that
\[
\sum_{i}b_{i}(z+\epsilon)^{i}=\sum_{i}b_{i}z^{i}+\sum_{i}a_{i}z^{i}\epsilon.
\]
This is equivalent to requiring $b_{1}=a_{0}$ and $b_{i}=0$ for
$i>1$. As a consequence, we can identify $\Hom^{G}(\Spec K_{f},V)$
with 
\[
\{(a,b+az)\in K_{f}^{2}\mid a,b\in k((t))\}\cong k((t))^{2}.
\]
For $f\ne0$, if we extend the order function on $k((t))$ to $K_{f}$
as a valuation, then $\ord(z)=\frac{\ord(f)}{p}$. Therefore, with
the above identification, $(a,b+az)\in\Hom^{G}(\Spec K_{f},V)$ lies
in $\Hom^{G}(\Spec\cO_{f},V)$ if and only if $\ord(a)\ge0$ and $\ord(b+az)\ge0$.
Since $\ord(f)$ is coprime with $p$, we have $\Z\ni\ord(b)\neq\ord(az)\notin\Z$
and therefore 
\[
\ord(b+az)=\min\{\ord(b),\ord(az)\}
\]
Thus the two conditions translate into $\ord(a)\ge\max\{0,\lceil-\ord(f)/p\rceil\}$
and $\ord(b)\ge0$. In conclusion, for every $f$, the set $\Hom^{G}(\Spec\cO_{f},V)$
is identified with the following subset of $\cO_{f}^{2}$,
\[
\{(a,b+az)\in\cO_{f}^{2}\mid a\in t^{s_{f}}\cdot k[[t]],\,b\in k[[t]]\}\quad(s_{f}:=\max\{0,\lceil-\ord(f)/p\rceil\}=\sht'(f)+2).
\]
We then identify $\cJ_{\infty}\cV$ with 
\[
\bigsqcup_{f\in\Delta_{G}}t^{s_{f}}k[[t]]\oplus k[[t]]
\]
and write its elements as triples $(f,a,b)$. For $m\in\ZZ_{\ge0}$,
the image of $\Hom^{G}(\Spec\cO_{f},V)$ under the natural map $\cO_{f}^{2}\to\cO_{f}^{2}/t^{m+1}\cO_{f}^{2}$,
which we denote by $\Hom^{G}(\Spec\cO_{f},V)_{m}$, coincides with
the image of the injective map
\[
\frac{t^{s_{f}}k[[t]]}{t^{s_{f}+m+1}k[[t]]}\oplus\frac{k[[t]]}{t^{m+1}k[[t]]}\arr\cO_{f}^{2}/t^{m+1}\cO_{f}^{2},\,(a,b)\mapsto(a,b+az).
\]

We define a motivic measure $\mu_{\cV}$ on $\cJ_{\infty}\cV$ as
follows. For $m,n\in\ZZ_{\ge0}$, let 
\[
\cJ_{m,n}\cV:=\bigsqcup_{f\in\Delta_{G,n}}\frac{t^{s_{f}}k[[t]]}{t^{s_{f}+m+1}k[[t]]}\oplus\frac{k[[t]]}{t^{m+1}k[[t]]}.
\]
Here $s_{f}=s_{g}$ where $g\in\Delta_{G}$ is any lift of $f\in\Delta_{G,n}$.
For integers $n,j$ with $np\ge j$, let $\Delta_{G,n}^{\ge j}\subset\Delta_{G,n}$
be the subspace of $f\in\Delta_{G,n}$ with $\ord(f)\ge j$. This
is an affine space of finite dimension and $\Delta_{G,n}$ is the
union of $\Delta_{G,n}^{\ge j}$, $j\le np$. This filtration also
allows to write $\shJ_{m,n}\shV$ as an increasing union of affine
spaces, so that the notion of constructible subsets and their measure
is well defined. For $m'\ge m$ and $n'\ge n$, the natural map $\cJ_{m',n'}\cV\to\cJ_{m,n}\cV$
is a trivial $\A_{k}^{2(m'-m)+(p-1)(n'-n)}$-bundle. Let $\pi_{m,n}\colon\cJ_{\infty}\cV\to\cJ_{m,n}\cV$
be the natural map. We say that a subset $C\subset\cJ_{\infty}\cV$
is a \emph{cylinder of level $(m,n)$} if $C=\pi_{m,n}^{-1}\pi_{m,n}(C)$
and $\pi_{m,n}(C)$ is a constructible subset of $\cJ_{m,n}\cV$.
Then we define the measure $\mu_{\cV}(C)$ of $C$ as 
\[
\mu_{\cV}(C):=[\pi_{m,n}(C)]\LL^{-2m-(p-1)n}.
\]
We can further extend this measure $\mu_{\cV}$ to measurable subsets,
following \cite[Appendix]{MR1905024}. A subset $C\subset\cJ_{\infty}\cV$
is \emph{measurable }if there exists a sequence of cylinders $C_{1},C_{2},\dots$
approximating $C$ (which means that there exists another sequence
$B_{1},B_{2},\dots$ of cylinders such that $\lim_{i\to\infty}\mu_{\cV}(B_{i})=0$
and for each $i$, the symmetric difference $C\triangle C_{i}=(C\cup C_{i})\setminus(C\cap C_{i})$
is contained in $B_{i}$). For a measurable subset $C$, we define
$\mu_{\cV}(C):=\lim_{i\to\infty}\mu_{\cV}(C_{i})$. A function $f\colon C\to\ZZ$
on a subset $C\subset\cJ_{\infty}\cV$ is said to be \emph{measurable}
if all fibers $f^{-1}(n)$ are measurable. The integral $\int_{C}\LL^{f}\,d\mu_{\cV}$
is then defined to be $\sum_{n\in\ZZ}[f^{-1}(n)]\LL^{n}$ in $\hat{\cM}'$,
provided that this infinite sum converges. 

The quotient variety $V/G$ has the coordinate ring $k[x,y^{p}]$.
The arc space $J_{\infty}(V/G)$ of $V/G$ is identified with $k[[t]]^{2}$
by looking at the images of $x$ and $y^{p}$ in $k[[t]]$. Similarly
the $m$-th jet scheme $J_{m}(V/G)$ is identified with $\left(k[[t]]/(t^{m+1})\right)^{2}$.
 Given an element of $\cJ_{\infty}\cV$ regarded as a $G$-equivariant
morphism $\Spec\cO_{f}\to V$, taking the $G$-quotient of the induced
morphism $\Spec K_{f}\to V$, we obtain a morphism $\Spec k((t))\to V/G$.
We easily see that this morphism extends to a morphism $\Spec k[[t]]\to V/G$.
Thus we obtain a map $\psi\colon\cJ_{\infty}\cV\to J_{\infty}(V/G)$.
In concrete terms, the map sends $(f,a,b)$ to $(a,b^{p}+fa^{p})$.
For $n\ge m\ge0$, the map $\psi$ induces map 
\[
\psi_{m,n}\colon\cJ_{m,n}\cV\to J_{m}(V/G),\,(f,a,b)\mapsto(a,b^{p}+fa^{p}).
\]

Let us take an element 
\[
(\alpha,\beta)\in\left(\frac{k[[t]]}{t^{m+1}k[[t]]}\right)^{\oplus2}=J_{m}(V/G).
\]
We will describe the fiber $\psi_{m,n}^{-1}((\alpha,\beta))$. Namely
we will describe the set of triples $(f,a,b)$ with $f\in\Delta_{G,n}$,
$a\in t^{s_{f}}k[[t]]/t^{s_{f}+m+1}k[[t]]$, $b\in k[[t]]/t^{m+1}k[[t]]$
such that $a=\alpha$ and $b^{p}+fa^{p}=\beta$ in $k[[t]]/t^{m+1}k[[t]]$.
Let us write $a=\sum_{i\le s_{f}+m}a_{i}t^{i}$, $b=\sum_{i\le m}b_{i}t^{i}$,
$f=\sum_{i\le np-1}f_{i}t^{i}$. The equality $a=\alpha$ determines
$a_{i}$, $i\le m$, requires that $s_{f}\leq\ord(\alpha)$ and put
no other constraint on $a_{i}$, $i>m$, $b_{i}$ or $f_{i}$. For
the equality $b^{p}+fa^{p}=\beta$, we note that $b^{p}$ (resp. $fa^{p}$)
has only terms of degrees divisible (resp. not divisible) by $p$.
Therefore this equality determines $b_{i}$, $i\le\lfloor m/p\rfloor$.
If $a$ is fixed and $m\ge p\cdot\ord(a)=p\cdot\ord(\alpha)$, then
the same equation determines $f_{i}$, $i\le m-p\cdot\ord(a)$, but
put no more constraint on $a_{i}$, $b_{i}$, $f_{i}$. In this case,
since the resulting $f$ is such that $fa^{p}$ has neither term of
negative degree nor term of degree divisible by $p$, it does follow
that $f\in\Delta_{G,n}$ and that $s_{f}\leq\ord(\alpha)$. Notice
moreover that, if $\beta'$ is the subsum of $\beta$ of degree coprime
with $p$, then $s_{f}$ is a function of $\ord(\beta')$, so that,
in particular, the number $s_{f}$ does not depends of the choice
of $(f,a,b)$ over $(\alpha,\beta)$. For simplicity, suppose $m=m'p$
for some $m'\in\NN$. As a consequence of the above computation, if
$\alpha\ne0$ and $m'\ge\ord(\alpha)$, then $\psi_{m,n}^{-1}((\alpha,\beta))$
is the affine space of dimension
\begin{align*}
 & \overset{\text{no. of free \ensuremath{a_{i}}}}{\overbrace{\{(s_{f}+m)-m\}}}+\overset{\text{no. of free \ensuremath{b_{i}}}}{\overbrace{(m-\lfloor m/p\rfloor)}}+\overset{\text{no. of free \ensuremath{f_{i}}}}{\overbrace{\{n-(m'-\ord(a))\}(p-1)}}\\
 & =s_{f}+(p-1)n+(p-1)\ord(a).
\end{align*}
We define functions
\begin{gather*}
s\colon\cJ_{\infty}\cV\to\ZZ,\,(f,a,b)\mapsto s_{f}=\sht'(f)+2,\\
\ord(x)\colon\cJ_{\infty}\cV\to\ZZ\sqcup\{\infty\},\,(f,a,b)\mapsto\ord(a).
\end{gather*}
From the above argument, $s$ is the composition of $\psi\colon\cJ_{\infty}\cV\to J_{\infty}(V/G)$
and a function $s'\colon J_{\infty}(V/G)\to\ZZ$, the latter having
cylindrical fibers. The function $\ord(x)$ on $\cJ_{\infty}\cV$
also factors through $\ord(x)\colon J_{\infty}(V/G)\to\ZZ\cup\{\infty\}$,
whose fibers are also cylinders except that $\ord(x)^{-1}(\infty)$
is a measurable subset of measure zero. Let $C\subset J_{\infty}(V/G)$
be a cylinder of level $m$. The inverse image $\psi^{-1}(C)$ is
a cylinder of level $(m,m)$ and 
\[
\pi_{m,m}(\psi^{-1}(C))=\psi_{m,m}^{-1}(\pi_{m}(C)).
\]
If $s$ and $\ord(x)$ take constant values $s_{0}$ and $r$ on $\psi^{-1}(C)$,
then 
\[
\mu_{V/G}(C)=\LL^{-s_{0}-(p-1)r}\mu_{\cV}(\psi^{-1}(C)).
\]
By a standard formal argument on measurable subsets (for instance,
see \cite[Proof of Th. 5.20]{MR3230848}), this equality is valid
also when $C$ is a measurable subset. Subdividing a given measurable
subset $C\subset J_{\infty}(V/G)$, we can reduce to the case where
$s$ and $\ord(x)$ are constant. These arguments lead to the change
of variables formula:
\begin{thm}
\label{thm:change-vars}Let $C\subset J_{\infty}(V/G)$ be a subset
and $F\colon C\arr\Z$ be a measurable function. Then
\[
\int_{C}\LL^{F}\,d\mu_{V/G}=\int_{\psi^{-1}(C)}\LL^{F\circ\psi-s-(p-1)\ord(x)}\,d\mu_{\cV}.
\]
In particular, 
\[
\LL^{2}=M_{\st}(V/G)=\int_{J_{\infty}(V/G)}1\,d\mu_{V/G}=\int_{\cJ_{\infty}\cV}\LL^{-s-(p-1)\ord(x)}\,d\mu_{\cV}.
\]
\end{thm}

In the last line $M_{\st}(V/G)=\int_{J_{\infty}(V/G)}1\,d\mu_{V/G}$
follows from definition and the fact that $V/G$ is smooth, while
$\int_{J_{\infty}(V/G)}1\,d\mu_{V/G}=\LL^{2}$ from the fact that
$J_{0}(V/G)=V/G=\A_{k}^{2}$. 

We describe here an alternative way to check the equality 
\[
\int_{\cJ_{\infty}\cV}\LL^{-s-(p-1)\ord(x)}\,d\mu_{\cV}=\LL^{2}.
\]
 The set 
\[
C_{\ge0,i}:=\{(f,a,b)\in\cJ_{\infty}\cV\mid\ord(f)\ge0,\,\ord(a)=i\}
\]
is a cylinder of level $(i,0)$ with $\pi_{i,0}(C_{\ge0,i})\cong\Gm\times\A_{k}^{i+1}$.
Therefore 
\[
\mu_{\cV}(C_{\geq0,i})=(\LL-1)\LL^{i+1}\LL^{-2i}=(\LL-1)\LL^{-i+1}
\]
Their disjoint union is $C_{\ge0}:=\{(f,a,b)\mid\ord(f)\ge0\}$ and
\begin{align*}
\int_{C_{\ge0}}\LL^{-s-(p-1)\ord(x)}\,d\mu_{\cV} & =\sum_{i\ge0}\mu_{\cV}(C_{0,i})\LL^{-(p-1)i}\\
 & =\frac{\LL^{2}-\LL}{1-\LL^{-p}}.
\end{align*}
For $j=-(pd+e)<0$ with $d\in\NN$ and $1\le e\le p-1$ and for $i\in\NN$,
let $C_{j,i}:=\{(f,a,b)\in\cJ_{\infty}\cV\mid\ord(f)=j,\,\ord(a)=s_{f}+i\}$.
This set is a cylinder of level $(i,0)$ such that 
\[
\pi_{i,0}(C_{j,i})\cong\overset{f}{\overbrace{\Gm\times\A_{k}^{d(p-1)+e-1}}}\times\overset{a}{\overbrace{\Gm}}\times\overset{b}{\overbrace{\A_{k}^{i+1}}}.
\]
Thus 
\[
\mu_{\cV}(C_{j,i})=(\LL-1)^{2}\LL^{-i+d(p-1)+e}.
\]
The disjoint union of all the $C_{j,i}$ is $C_{<0}=\{(f,a,b)\mid\ord(f)<0\}$.
Then
\begin{align*}
\int_{C_{<0}}\LL^{-s-(p-1)(x)}\,d\mu_{\cV} & =\sum_{1\le e\le p-1}\sum_{d\in\NN}\sum_{i\in\NN}(\LL-1)^{2}\LL^{-i+d(p-1)+e}\times\LL^{-(d+1)-(p-1)(d+1+i)}\\
 & =\sum_{1\le e\le p-1}\sum_{d\in\NN}\sum_{i\in\NN}(\LL-1)^{2}\LL^{-pi+e-d-p}\\
 & =(\LL-1)^{2}\LL^{-p}\frac{\LL^{1}+\cdots+\LL^{p-1}}{(1-\LL^{-p})(1-\LL^{-1})}\\
 & =\frac{\LL-\LL^{2-p}}{1-\LL^{-p}}.
\end{align*}
It follows that 
\[
\int_{\cJ_{\infty}\cV}\LL^{-s-(p-1)\ord(x)}\,d\mu_{\cV}=\frac{\LL^{2}-\LL}{1-\LL^{-p}}+\frac{\LL-\LL^{2-p}}{1-\LL^{-p}}=\LL^{2}.
\]
\begin{rem}
It is natural to see the function $(p-1)\ord(x)$ in Theorem \ref{thm:change-vars}
as a counterpart of the function $\ord(y^{p-1})=(p-1)\ord(y)$ in
the change of variables formula (\ref{eq:ch-vars-1}) for the case
of $H$. The latter is the order function associated to the Jacobian
ideal $(y^{p-1})\subset k[x,y]$ of the map $W\to W/H$. It is a natural
problem, how to derive the ideal $(x^{p-1})$ as the ``Jacobian ideal''
of $V\to V/G$, a map not generically \'{e}tale. 
\end{rem}

\appendix

\section{Representation theory of $\alpha_{p}$}

In this appendix we recall the representation theory of the group
scheme $\alpha_{p}$ over $\F_{p}$.

Given an $\F_{p}$-algebra $A$ we denote by $\Mod^{\alpha_{p}}A$
the category of $A$-modules with an action of $\alpha_{p}\times A$,
or, equivalently, a coaction of the Hopf algebra $A[\alpha_{p}]=A[\varepsilon]$.
We introduce also the category $\Mod^{nil}A$ of pairs $(M,\xi)$
where $M$ is an $A$-module and $\xi\colon M\arr M$ is an $A$-linear
map which is $p$-nilpotent, that is $\xi^{p}=0$. Given $(M,\xi)\in\Mod^{nil}A$
we define
\[
\exp(\xi\varepsilon)=\sum_{i=0}^{p-1}\frac{\xi^{i}\varepsilon^{i}}{i!}\colon M\arr M\otimes\F_{p}[\varepsilon]
\]
\begin{prop}
The functor
\[
\Mod^{nil}(A)\arr\Mod^{\alpha_{p}}(A)\comma(M,\xi)\longmapsto(M,\exp(\xi\varepsilon))
\]
is well defined and an equivalence of categories. Moreover for $(M,\xi),$
$(N,\eta)\in\Mod^{nil}(A)$ we have $M^{\alpha_{p}}=\Ker(\xi)$ and
that 
\[
\xi\otimes\id_{M}+\id_{N}\otimes\eta\colon M\otimes N\arr M\otimes N
\]
corresponds to the tensor product of $\alpha_{p}$-modules.

If $B$ is an $A$-algebra and $(B,\xi)\in\Mod^{nil}(A)$ then $\alpha_{p}$
acts on the $A$-algebra $B$, that is $A\arr B$ and $B\otimes_{A}B\arr B$
are $\alpha_{p}$-equivariant, if and only if $\xi\colon B\arr B$
is an $A$-derivation.
\end{prop}

\begin{proof}
Let $M$ be an $A$-module and $\phi\colon M\arr M\otimes k[\varepsilon]$
be an $A$-linear map. The map $\phi$ can be written as 
\[
\phi=\sum_{i=0}^{p-1}\phi_{i}\varepsilon^{i}\text{ for }A\text{-linear maps }\phi_{i}\colon M\arr M
\]
The map $\phi$ must satisfy the following two conditions in order
to be an $\alpha_{p}$-action: $(\id_{M}\otimes z)\circ\phi=\id_{M}\colon M\arr M$,
where $z\colon\F_{p}[\varepsilon]\arr\F_{p}$, $z(\varepsilon)=0$
is the $0$-section and $(\id_{M}\otimes\Delta)\circ\phi=(\phi\otimes\id_{\FF_{p}[\epsilon]})\circ\phi\colon M\arr M\otimes\F_{p}[\varepsilon]\otimes\F_{p}[\varepsilon]$,
where $\Delta\colon\F_{p}[\varepsilon]\arr\F_{p}[\varepsilon]\otimes\F_{p}[\varepsilon]$,
$\Delta(\varepsilon)=\varepsilon\otimes1+1\otimes\varepsilon$ is
the comultiplication. Those conditions translate into
\[
\phi_{0}=\id_{M}\text{ and }\phi_{i}\phi_{j}=\binom{i+j}{i}\phi_{i+j}\text{ for }0\leq i,j<p
\]
and into $\phi_{1}^{p}=0$, $\phi=\exp(\phi_{1}\varepsilon)$. This
easily prove the equivalence in the statement.

The trivial $\alpha_{p}$-action on $A$ corresponds to the nilpotent
endomorphism $A\arrdi 0A$. Thus
\[
M^{\alpha_{p}}=\Hom^{\alpha_{p}}(A,M)=\Hom_{\Mod^{nil}(A)}((A,0),(M,\xi))=\Ker(\xi)
\]
The claim about the tensor product follows from a direct check.

Consider now the last statement. The map $\iota\colon A\arr B$ is
$\alpha_{p}$-equivariant if $\iota$ is compatible with the nilpotent
endomorphisms $A\xlongrightarrow{0}A$ and $\xi$ if and only if $\xi(\iota(a))=0$
for $a\in A$. From the assertion of the tensor product, the map $B\otimes_{A}B\arr B$
is $\alpha_{p}$-equivariant if and only if this map is compatible
with $\xi\colon B\longrightarrow B$ and $\xi\otimes\id_{B}+\id_{B}\otimes\xi\colon B\otimes_{A}B\to B\otimes_{A}B$
if and only if $\xi$ satisfies the Leibniz rule. This ends the proof.
\end{proof}
\begin{example}
Let $(M,\xi)\in\Mod^{nil}A$. Then $\alpha_{p}$ acts on the $A$-algebra
$\Sym(M)$ and the corresponding $p$-nilpotent endomorphism $\xi_{*}\colon\Sym(M)\arr\Sym(M)$
is the unique $A$-derivation such that $(\xi_{*})_{|M}=\xi$.

In particular the corresponding $p$-nilpotent endomorphism $\xi_{n}\colon\Sym^{n}M\arr\Sym^{n}M$
is given by 
\[
\xi_{n}(m_{1}\cdots m_{n})=\xi(m_{1})m_{2}\cdots m_{n}+\cdots+m_{1}\cdots m_{n-1}\xi(m_{n})
\]
\end{example}

\begin{example}
Assume that $A=k$ is a field. Then any $p$-nilpotent endomorphism
$\xi\colon k^{n}\arr k^{n}$ can be put in Jordan form and, in this
case, this just means that all blocks have $0$ diagonal and have
size at most $p$. It follows that, up to isomorphisms, the $\alpha_{p}$-representions
over $k$ correspond bijectively to sequences $\bd=(d_{1},\dots,d_{l})$
with $1\leq d_{i}\leq p$, $d_{i}\geq d_{i+1}$ and $l\in\N$.\bibliographystyle{plain}
\bibliography{ramif-gps}
\end{example}

\end{document}

%% file: packages_and_functions.for_preamble.tex
\usepackage{extarrow}
\usepackage{enumerate}
\usepackage{fontenc}
\usepackage[T1]{fontenc}
\usepackage{graphicx}
\usepackage[colorlinks=true, linkcolor=blue]{hyperref}
\usepackage{latexsym}
\usepackage{mathrsfs}
\usepackage{mathtext}
\usepackage{stmaryrd}
\usepackage{tikz}
\usepackage{textcomp}
\usepackage{upgreek}
\usepackage{xy}

\usetikzlibrary{arrows}
\usetikzlibrary{matrix}
\usetikzlibrary{shapes}
\usetikzlibrary{snakes}
\usetikzlibrary{matrix}

\DeclareMathOperator{\Add}{\textup{Add}}
\DeclareMathOperator{\Aff}{\textup{Aff}}
\DeclareMathOperator{\Alg}{\textup{Alg}}
\DeclareMathOperator{\Ann}{\textup{Ann}}
\DeclareMathOperator{\Arr}{\textup{Arr}}
\DeclareMathOperator{\Art}{\textup{Art}}
\DeclareMathOperator{\Ass}{\textup{Ass}}
\DeclareMathOperator{\Aut}{\textup{Aut}}
\DeclareMathOperator{\Autsh}{\underline{\textup{Aut}}}
\DeclareMathOperator{\Bi}{\textup{B}}
\DeclareMathOperator{\CAdd}{\textup{CAdd}}
\DeclareMathOperator{\CAlg}{\textup{CAlg}}
\DeclareMathOperator{\CMon}{\textup{CMon}}
\DeclareMathOperator{\CPMon}{\textup{CPMon}}
\DeclareMathOperator{\CRings}{\textup{CRings}}
\DeclareMathOperator{\CSMon}{\textup{CSMon}}
\DeclareMathOperator{\CaCl}{\textup{CaCl}}
\DeclareMathOperator{\Cart}{\textup{Cart}}
\DeclareMathOperator{\Cl}{\textup{Cl}}
\DeclareMathOperator{\Coh}{\textup{Coh}}
\DeclareMathOperator{\Coker}{\textup{Coker}}
\DeclareMathOperator{\Cov}{\textup{Cov}}
\DeclareMathOperator{\Der}{\textup{Der}}
\DeclareMathOperator{\Div}{\textup{Div}}
\DeclareMathOperator{\End}{\textup{End}}
\DeclareMathOperator{\Endsh}{\underline{\textup{End}}}
\DeclareMathOperator{\Ext}{\textup{Ext}}
\DeclareMathOperator{\Extsh}{\underline{\textup{Ext}}}
\DeclareMathOperator{\FAdd}{\textup{FAdd}}
\DeclareMathOperator{\FCoh}{\textup{FCoh}}
\DeclareMathOperator{\FGrad}{\textup{FGrad}}
\DeclareMathOperator{\FLoc}{\textup{FLoc}}
\DeclareMathOperator{\FMod}{\textup{FMod}}
\DeclareMathOperator{\FPMon}{\textup{FPMon}}
\DeclareMathOperator{\FRep}{\textup{FRep}}
\DeclareMathOperator{\FSMon}{\textup{FSMon}}
\DeclareMathOperator{\FVect}{\textup{FVect}}
\DeclareMathOperator{\Fibr}{\textup{Fibr}}
\DeclareMathOperator{\Fix}{\textup{Fix}}
\DeclareMathOperator{\Fl}{\textup{Fl}}
\DeclareMathOperator{\Fr}{\textup{Fr}}
\DeclareMathOperator{\Funct}{\textup{Funct}}
\DeclareMathOperator{\GAlg}{\textup{GAlg}}
\DeclareMathOperator{\GExt}{\textup{GExt}}
\DeclareMathOperator{\GHom}{\textup{GHom}}
\DeclareMathOperator{\GL}{\textup{GL}}
\DeclareMathOperator{\GMod}{\textup{GMod}}
\DeclareMathOperator{\GRis}{\textup{GRis}}
\DeclareMathOperator{\GRiv}{\textup{GRiv}}
\DeclareMathOperator{\Gal}{\textup{Gal}}
\DeclareMathOperator{\Gl}{\textup{Gl}}
\DeclareMathOperator{\Grad}{\textup{Grad}}
\DeclareMathOperator{\Hilb}{\textup{Hilb}}
\DeclareMathOperator{\Hl}{\textup{H}}
\DeclareMathOperator{\Hom}{\textup{Hom}}
\DeclareMathOperator{\Homsh}{\underline{\textup{Hom}}}
\DeclareMathOperator{\ISym}{\textup{Sym}^*}
\DeclareMathOperator{\Imm}{\textup{Im}}
\DeclareMathOperator{\Irr}{\textup{Irr}}
\DeclareMathOperator{\Iso}{\textup{Iso}}
\DeclareMathOperator{\Isosh}{\underline{\textup{Iso}}}
\DeclareMathOperator{\Ker}{\textup{Ker}}
\DeclareMathOperator{\LAdd}{\textup{LAdd}}
\DeclareMathOperator{\LAlg}{\textup{LAlg}}
\DeclareMathOperator{\LMon}{\textup{LMon}}
\DeclareMathOperator{\LPMon}{\textup{LPMon}}
\DeclareMathOperator{\LRings}{\textup{LRings}}
\DeclareMathOperator{\LSMon}{\textup{LSMon}}
\DeclareMathOperator{\Left}{\textup{L}}
\DeclareMathOperator{\Lex}{\textup{Lex}}
\DeclareMathOperator{\Loc}{\textup{Loc}}
\DeclareMathOperator{\M}{\textup{M}}
\DeclareMathOperator{\ML}{\textup{ML}}
\DeclareMathOperator{\MLex}{\textup{MLex}}
\DeclareMathOperator{\Map}{\textup{Map}}
\DeclareMathOperator{\Mod}{\textup{Mod}}
\DeclareMathOperator{\Mon}{\textup{Mon}}
\DeclareMathOperator{\Ob}{\textup{Ob}}
\DeclareMathOperator{\Obj}{\textup{Obj}}
\DeclareMathOperator{\PDiv}{\textup{PDiv}}
\DeclareMathOperator{\PGL}{\textup{PGL}}
\DeclareMathOperator{\PML}{\textup{PML}}
\DeclareMathOperator{\PMLex}{\textup{PMLex}}
\DeclareMathOperator{\PMon}{\textup{PMon}}
\DeclareMathOperator{\Pic}{\textup{Pic}}
\DeclareMathOperator{\Picsh}{\underline{\textup{Pic}}}
\DeclareMathOperator{\Pro}{\textup{Pro}}
\DeclareMathOperator{\Proj}{\textup{Proj}}
\DeclareMathOperator{\QAdd}{\textup{QAdd}}
\DeclareMathOperator{\QAlg}{\textup{QAlg}}
\DeclareMathOperator{\QCoh}{\textup{QCoh}}
\DeclareMathOperator{\QMon}{\textup{QMon}}
\DeclareMathOperator{\QPMon}{\textup{QPMon}}
\DeclareMathOperator{\QRings}{\textup{QRings}}
\DeclareMathOperator{\QSMon}{\textup{QSMon}}
\DeclareMathOperator{\R}{\textup{R}}
\DeclareMathOperator{\Rep}{\textup{Rep}}
\DeclareMathOperator{\Rings}{\textup{Rings}}
\DeclareMathOperator{\Riv}{\textup{Riv}}
\DeclareMathOperator{\SFibr}{\textup{SFibr}}
\DeclareMathOperator{\SMLex}{\textup{SMLex}}
\DeclareMathOperator{\SMex}{\textup{SMex}}
\DeclareMathOperator{\SMon}{\textup{SMon}}
\DeclareMathOperator{\SchI}{\textup{SchI}}
\DeclareMathOperator{\Sh}{\textup{Sh}}
\DeclareMathOperator{\Soc}{\textup{Soc}}
\DeclareMathOperator{\Spec}{\textup{Spec}}
\DeclareMathOperator{\Specsh}{\underline{\textup{Spec}}}
\DeclareMathOperator{\Stab}{\textup{Stab}}
\DeclareMathOperator{\Supp}{\textup{Supp}}
\DeclareMathOperator{\Sym}{\textup{Sym}}
\DeclareMathOperator{\TMod}{\textup{TMod}}
\DeclareMathOperator{\Top}{\textup{Top}}
\DeclareMathOperator{\Tor}{\textup{Tor}}
\DeclareMathOperator{\Vect}{\textup{Vect}}
\DeclareMathOperator{\alt}{\textup{ht}}
\DeclareMathOperator{\car}{\textup{char}}
\DeclareMathOperator{\codim}{\textup{codim}}
\DeclareMathOperator{\degtr}{\textup{degtr}}
\DeclareMathOperator{\depth}{\textup{depth}}
\DeclareMathOperator{\divis}{\textup{div}}
\DeclareMathOperator{\et}{\textup{et}}
\DeclareMathOperator{\ffpSch}{\textup{ffpSch}}
\DeclareMathOperator{\h}{\textup{h}}
\DeclareMathOperator{\ilim}{\displaystyle{\lim_{\longrightarrow}}}
\DeclareMathOperator{\ind}{\textup{ind}}
\DeclareMathOperator{\indim}{\textup{inj dim}}
\DeclareMathOperator{\lf}{\textup{LF}}
\DeclareMathOperator{\op}{\textup{op}}
\DeclareMathOperator{\ord}{\textup{ord}}
\DeclareMathOperator{\pd}{\textup{pd}}
\DeclareMathOperator{\plim}{\displaystyle{\lim_{\longleftarrow}}}
\DeclareMathOperator{\pr}{\textup{pr}}
\DeclareMathOperator{\pt}{\textup{pt}}
\DeclareMathOperator{\rk}{\textup{rk}}
\DeclareMathOperator{\tr}{\textup{tr}}
\DeclareMathOperator{\type}{\textup{r}}
\DeclareMathOperator*{\colim}{\textup{colim}}

%% file: packages_and_functions.tex
\global\long\def\A{\mathbb{A}}

\global\long\def\Ab{(\textup{Ab})}

\global\long\def\C{\mathbb{C}}

\global\long\def\Cat{(\textup{cat})}

\global\long\def\Di#1{\textup{D}(#1)}

\global\long\def\E{\mathcal{E}}

\global\long\def\F{\mathbb{F}}

\global\long\def\GCov{G\textup{-Cov}}

\global\long\def\Gcat{(\textup{Galois cat})}

\global\long\def\Gfsets#1{#1\textup{-fsets}}

\global\long\def\Gm{\mathbb{G}_{m}}

\global\long\def\GrCov#1{\textup{D}(#1)\textup{-Cov}}

\global\long\def\Grp{(\textup{Grps})}

\global\long\def\Gsets#1{(#1\textup{-sets})}

\global\long\def\HCov{H\textup{-Cov}}

\global\long\def\MCov{\textup{D}(M)\textup{-Cov}}

\global\long\def\MHilb{M\textup{-Hilb}}

\global\long\def\N{\mathbb{N}}

\global\long\def\PGor{\textup{PGor}}

\global\long\def\PGrp{(\textup{Profinite Grp})}

\global\long\def\PP{\mathbb{P}}

\global\long\def\Pj{\mathbb{P}}

\global\long\def\Q{\mathbb{Q}}

\global\long\def\RCov#1{#1\textup{-Cov}}

\global\long\def\RR{\mathbb{R}}

\global\long\def\WW{\textup{W}}

\global\long\def\Z{\mathbb{Z}}

\global\long\def\acts{\curvearrowright}

\global\long\def\alA{\mathscr{A}}

\global\long\def\alB{\mathscr{B}}

\global\long\def\arr{\longrightarrow}

\global\long\def\arrdi#1{\xlongrightarrow{#1}}

\global\long\def\catC{\mathscr{C}}

\global\long\def\catD{\mathscr{D}}

\global\long\def\catF{\mathscr{F}}

\global\long\def\catG{\mathscr{G}}

\global\long\def\comma{,\ }

\global\long\def\covU{\mathcal{U}}

\global\long\def\covV{\mathcal{V}}

\global\long\def\covW{\mathcal{W}}

\global\long\def\duale#1{{#1}^{\vee}}

\global\long\def\fasc#1{\widetilde{#1}}

\global\long\def\fsets{(\textup{f-sets})}

\global\long\def\iL{r\mathscr{L}}

\global\long\def\id{\textup{id}}

\global\long\def\la{\langle}

\global\long\def\odi#1{\mathcal{O}_{#1}}

\global\long\def\ra{\rangle}

\global\long\def\rig{\mathbin{\!\!\pmb{\fatslash}}}

\global\long\def\set{(\textup{Sets})}

\global\long\def\shA{\mathcal{A}}

\global\long\def\shB{\mathcal{B}}

\global\long\def\shC{\mathcal{C}}

\global\long\def\shD{\mathcal{D}}

\global\long\def\shE{\mathcal{E}}

\global\long\def\shF{\mathcal{F}}

\global\long\def\shG{\mathcal{G}}

\global\long\def\shH{\mathcal{H}}

\global\long\def\shI{\mathcal{I}}

\global\long\def\shJ{\mathcal{J}}

\global\long\def\shK{\mathcal{K}}

\global\long\def\shL{\mathcal{L}}

\global\long\def\shM{\mathcal{M}}

\global\long\def\shN{\mathcal{N}}

\global\long\def\shO{\mathcal{O}}

\global\long\def\shP{\mathcal{P}}

\global\long\def\shQ{\mathcal{Q}}

\global\long\def\shR{\mathcal{R}}

\global\long\def\shS{\mathcal{S}}

\global\long\def\shT{\mathcal{T}}

\global\long\def\shU{\mathcal{U}}

\global\long\def\shV{\mathcal{V}}

\global\long\def\shW{\mathcal{W}}

\global\long\def\shX{\mathcal{X}}

\global\long\def\shY{\mathcal{Y}}

\global\long\def\shZ{\mathcal{Z}}

\global\long\def\st{\ | \ }

\global\long\def\stA{\mathcal{A}}

\global\long\def\stB{\mathcal{B}}

\global\long\def\stC{\mathcal{C}}

\global\long\def\stD{\mathcal{D}}

\global\long\def\stE{\mathcal{E}}

\global\long\def\stF{\mathcal{F}}

\global\long\def\stG{\mathcal{G}}

\global\long\def\stH{\mathcal{H}}

\global\long\def\stI{\mathcal{I}}

\global\long\def\stJ{\mathcal{J}}

\global\long\def\stK{\mathcal{K}}

\global\long\def\stL{\mathcal{L}}

\global\long\def\stM{\mathcal{M}}

\global\long\def\stN{\mathcal{N}}

\global\long\def\stO{\mathcal{O}}

\global\long\def\stP{\mathcal{P}}

\global\long\def\stQ{\mathcal{Q}}

\global\long\def\stR{\mathcal{R}}

\global\long\def\stS{\mathcal{S}}

\global\long\def\stT{\mathcal{T}}

\global\long\def\stU{\mathcal{U}}

\global\long\def\stV{\mathcal{V}}

\global\long\def\stW{\mathcal{W}}

\global\long\def\stX{\mathcal{X}}

\global\long\def\stY{\mathcal{Y}}

\global\long\def\stZ{\mathcal{Z}}

\global\long\def\then{\ \Longrightarrow\ }

\global\long\def\L{\textup{L}}

\global\long\def\l{\textup{l}}